\pdfoutput=1
\documentclass[journal]{IEEEtran}
\usepackage{amsmath,amsfonts,amssymb,amsthm}

\usepackage{cite}
\usepackage{graphicx}
\usepackage{float}
\usepackage{caption}
\usepackage{bm} 
\usepackage{dsfont}
\usepackage{wasysym}
\usepackage{algorithmicx}
\usepackage{algpseudocode}
\usepackage{algorithm}
\algdef{SE}[DOUNTIL]{Do}{Until}{\algorithmicdo}[1]{\textbf{until}\ #1}%
\usepackage{setspace}
\usepackage{accents}
\usepackage{nameref}

\usepackage{etoolbox}
\usepackage{enumerate}
\usepackage{subcaption}
\newcommand{\ubar}[1]{\text{\b{$#1$}}}
\usepackage{nicefrac}
\usepackage{multirow}
\usepackage{hyperref}
\usepackage{cleveref}
\newtheorem{theorem}{Theorem}[section]
\newtheorem{lemma}[theorem]{Lemma}

\newtheorem{definition}[theorem]{Definition}
\newtheorem{corollary}[theorem]{Corollary}

\begin{document}	
	\renewcommand{\qed}{\null\hfill\IEEEQEDclosed}
	\raggedbottom
	\allowdisplaybreaks
	\bstctlcite{MyBSTcontrol}
	\title{Minimising Unserved Energy\\Using Heterogeneous Storage Units}%
	\author{Michael P. Evans, Simon H. Tindemans, \IEEEmembership{Member, IEEE} and David Angeli, \IEEEmembership{Fellow, IEEE}
		\thanks{M.P. Evans and D. Angeli are with the Department of Electrical and Electronic Engineering, Imperial College London, UK (email: m.evans16@imperial.ac.uk; d.angeli@imperial.ac.uk). D. Angeli is in addition with the Department of Information Engineering, University of Florence, Italy. S.H. Tindemans is with the Department of Electrical Sustainable Energy, TU Delft, Netherlands (email: s.h.tindemans@tudelft.nl). This work was supported by EPSRC studentship 1688672.}
		\vspace{-1cm}}
	\date{}
	
	\markboth{Submitted to Transactions on Power Systems}{}
	\maketitle
	
	\begin{abstract}
		This paper considers the optimal dispatch of energy-constrained heterogeneous storage units to maximise security of supply. A policy, requiring no knowledge of the future, is presented and shown to minimise unserved energy during supply-shortfall events, regardless of the supply and demand profiles. It is accompanied by a graphical means to rapidly determine unavoidable energy shortfalls, which can then be used to compare different device fleets. The policy is well-suited for use within the framework of system adequacy assessment; for this purpose, a discrete time optimal policy is conceived, in both analytic and algorithmic forms, such that these results can be applied to discrete time systems and simulation studies. This is exemplified via a generation adequacy study of the British system.
	\end{abstract}
    
	
	\section{Introduction}
	\label{sec:intro}
	The generation mix is changing, in particular towards increased proliferation of renewable sources, which pose challenges as a result of the stochastic nature of their outputs. Energy storage offers a promising means of mitigating the effect of such fluctuations and, by shifting consumption in time, may also allow for a reduction in the total generation requirement of the network. As a result, recent years have seen an increasing abundance of storage devices connected to electricity networks. 
	
	Generation capacity adequacy can be interpreted as the sustenance of a specified upper limit on the risk of failure to meet demand. Such security of supply risks are multi-faceted, but commonly represented by a few metrics. Perhaps the two most common choices are expected energy-not-served (EENS) and loss-of-load expectation (LOLE). As argued in \cite{Edwards2017}, EENS offers improvements over LOLE as the total shortfall in energy is a more appropriate measure of adequacy than the total length of time for which any shortfall is experienced. Indeed, the European Commission reports that EENS, unlike LOLE, is able to encode the severity of the disruption and enable a monetisation of interruption costs \cite{EuCom}. System operators and regulators monitor these risk metrics, computed via probabilistic studies. With the increased implementation of energy storage, it is therefore imperative that these entities have a means of incorporating storage devices into such studies. This issue becomes more pressing when these studies are used in capacity markets, since the interpretation of the results can have significant financial implications. In the British (GB) market, for example, $\pounds 1.2$ billion of capacity was awarded in the 2016 T-4 auction \cite{ofgem}. In the GB system, the LOLE metric with a target of 3 hours per year is used to determine the overall capacity requirement. However, a new policy for computing the adequacy contribution of storage was adopted in 2018, in which the capacity value of units (represented by their \textit{de-rating factor}) is computed with respect to the EENS metric \cite{NG}. The underlying capacity adequacy studies must therefore run storage in an EENS-minimising fashion. The chosen policies for dispatch of storage devices in simulation studies are described by the GB TSO in \cite{NG}, but provided without a theoretical basis for this methodology. In this paper we derive the dispatch policy that results in the minimum energy shortfall, or energy-not-supplied (ENS) for any realised scenario; by definition, such a policy minimises the \emph{expectation} of ENS (EENS) across all possible scenarios. We also provide an algorithm that is suitable for EENS-minimising studies on the contribution of energy storage to system adequacy.
	
	There is significant literature in which numerical assessments are made of the ability of storage to replace conventional generation. A common approach is to use it as a buffer between committed and realised hourly outputs, of wind \cite{Hu2009} and solar PV \cite{Zhou2017} for example. The latter also performs an optimisation of standard form with the intention of minimising peak load. In \cite{Zhou2015}, the same authors incorporate into their objective function demand response (DR) of varying payback requirements coupled with storage. \cite{Sioshansi2014} assumes that the storage operator aims only to maximise profits, with perfect foresight and no direct consideration of adequacy requirements. Clearly, prior literature covers a range of assumed objectives for a storage operator, and lacks a consistent operational strategy for devices. We argue that minimising EENS is a suitable objective, which is also addressed in \cite{Edwards2017}, albeit for a single store only. 
	
	We consider the task of dispatching a fleet of heterogeneous storage devices so as to best support the grid under supply-shortfall conditions. It is of course crucial that diversity in the available devices be accounted for, so that multiple market players might supply support services. This however poses additional challenges as it is not straightforward to aggregate a fleet of this form. Moreover, in reality many of the storage resources that are connected onto electricity networks are energy-constrained, for example due to physical capacity limits or operational limits set by users, which forms a key consideration in our analysis. 
	
	This paper builds directly on \cite{Evans2017}, which considered the dispatch policy that achieves maximal failure-time, and \cite{Evans2018}, which presented methods of capability assessments for a fleet of storage units. We here extend these results with the  objective of maximising the contribution of energy storage to system adequacy. The contributions of this paper are as follows:
	\begin{itemize}
		\item We extend the aforementioned dispatch policy to include scenarios where the storage fleet is unable to meet demand. 
		\item We show that this policy results in the minimum achievable ENS for a fleet of heterogeneous storage units, irrespective of the actual received reference.
		\item We define an extension to the policy that includes charging operation.
		\item For comparison purposes, we define a peak-shaving policy that achieves the same minimal ENS by relying on perfect foresight.
		\item We present an immediate means of determining the minimum ENS for a heterogeneous fleet using the $E$-$p$ transform of \cite{Evans2018}.
		\item We provide a discrete time algorithm that implements a minimum ENS policy for time-resolved generation adequacy studies.
	\end{itemize}
	
	\noindent It follows that the policy described in this paper can be used to minimise ENS, irrespective of the risk appetite of the storage operator. The discrete time version can therefore be used in simulation studies which include storage to improve security of supply. We illustrate this fact by means of a case study for the GB system with heterogeneous units.
	
	\section{Problem Framework}
	We consider the centralised control of the aggregate response of a fleet of storage units, and for the purposes of description assume that this is undertaken by an aggregator. We assume that the aggregator receives a periodically-updated request from the network operator and accordingly must decide how to deploy its fleet of resources, without knowledge of future request signals. We focus on the decision making of the aggregator, and consider how best to dispatch devices in the presence of uncertain demand. This has been explicitly considered in \cite{Bejan2012,Gast2014,Cruise2015}, applied in the latter to arbitrage as well as system support (balancing services under supply-shortfall conditions). We utilise a similar methodology but take as our objective the provision of maximal contribution to system adequacy.
	
	Ancillary service provision from storage devices has been previously studied in the literature, across a range of devices. These include diesel generators \cite{Wang2015,Mazidi2014}, electric vehicles (EVs) \cite{Shao2012,Sortomme2012,Wenzel2017} and home storage devices \cite{Li2011,Tsui2012,Wang2013}; we here compose a general integrator model, which then allows us to accommodate any such devices. In contrast to the majority of this prior work, \cite{Mazidi2014,Shao2012,Sortomme2012,Wenzel2017,Li2011,Tsui2012,Wang2013}, we consider the aggregate system support capability of the storage units to be of paramount importance, so that the ability to satisfy the power request takes precedence over other objectives. Note that in an economic context this is equivalent to associating a very high cost with failure to meet demand. This allows us to study properties of the system in a general sense, without considering price dynamics in detail. However, a comparable policy to that which we propose, applied to a continuum of devices, is also relevant in such settings \cite{dePaola2017} and enables the coordination of devices within areas of flat prices. \cite{Wang2015} is concerned with trading off the use of a battery and a diesel generator to support renewable resources, and uses constraints to ensure no mismatch between supply and demand. We allow for such a mismatch, with the objective of minimising total energy shortfall over time. We also assume the absence of cross-charging between devices, which corresponds to a regime in which operational losses are minimised. We apply this modelling to the following two cases of interest:
	\begin{enumerate}
		\item An aggregator is contracted to provide system support, receiving periodically updated targets, and is penalised for ENS.
		\item System adequacy studies. For this, we simply remove the role of the aggregator and assume that the system operator is directly tasked with dispatching resources to meet excess demand. We use generic storage units but assume that when a shortfall occurs we are operating the devices in discharge-only mode. 
	\end{enumerate}
	
	\section{Mathematical formulation}
	\subsection{Problem description}
	\noindent We denote by $\mathcal{N}\doteq\{D_1,D_2,...,D_n\}$ the set of energy-constrained storage devices available to the aggregator. We do not impose any restrictions on homogeneity of devices and allow each to have a unique discharging efficiency. For convenience we incorporate this into the model implicitly by considering the extractable energy of each device, $e_i(t)$. We choose the power delivered by each device to be the control input $u_i(t)$, and assume that this is measured externally so that efficiency is once again accounted for. This leads to integrator dynamics on the energy of each device $\dot{e}_i(t)=-u_i(t),$ subject to the assumed physical constraint $e_i(t)\geq0$. As discussed above, we restrict our devices to discharging operation only, so that the power of each device is constrained as $u_i(t)\in[0,\bar{p}_i]$, in which $\bar{p}_i$ denotes the maximum discharge rate of device $D_i$, and with the  convention that discharge rates are positive. We define the \textit{time-to-go} of device $D_i$ to be the time remaining for which this device can run at its maximum power, i.e. $x_i(t) \doteq {e_i(t)}/{\bar{p}_i}$, and represent the state of each device by its time-to-go. We then form state, input and maximum power vectors as
	\begin{alignat}{3}
	x(t)&\doteq \begin{bmatrix}
	x_1(t)\ \dots\ x_n(t)
	\end{bmatrix}^T,\\
	u(t) &\doteq \begin{bmatrix}
	u_1(t)\ \dots\ u_n(t)
	\end{bmatrix}^T,\\
	\bar{p}&\doteq\begin{bmatrix}
	\bar{p}_1\ \dots\	\bar{p}_n
	\end{bmatrix}^T,
	\end{alignat}
	respectively, so that we can write our dynamics in matrix form as $\dot{x}(t)=-P^{-1}u(t)$, in which $P \doteq \text{diag}(\bar{p})$. We finally define the state space as $\mathcal{X}\doteq [0,+\infty)^n$, and form the product set of our constraints on all the inputs,
	\begin{equation}
	\mathcal{U}_{\bar{p}} \doteq [0,\bar{p}_1] \times [0,\bar{p}_2] \times ... \times [0,\bar{p}_n],
	\end{equation}	
	allowing us to write our input constraints as $u(t)\in\mathcal{U}_{\bar{p}}$.
	
	We denote by $P^r\colon [0,+\infty) \mapsto[0,+\infty)$ a power reference signal received by the aggregator, and in addition denote a truncated trajectory of such a signal as
	\begin{equation}
	P^r_{[t_0,t)}\doteq\left\{
	\begin{array}{@{}ll@{}}
	P^r(\tau), & \text{if}\ \tau\in[t_0,t) \\
	0, & \text{otherwise}.
	\end{array}\right.
	\end{equation} 
	We say that a reference that is satisfiable for all time without violating any constraints is \textit{feasible}, and define the set of such signals as follows:
	\begin{definition}
		\label{def:feasible}
		The set of feasible power reference signals, for a system with maximum power vector $\bar{p}$ and initial state $x=x(0)$, is defined as
		\begin{equation*}
		\begin{aligned}
		\mathcal{F}_{\bar{p},x} \doteq \big\{&P^r(\cdot) \colon \exists u(\cdot),\ z(\cdot)\colon \forall t\geq 0,\ 1^Tu(t)=P^r(t),\\
		&u(t)\in \mathcal{U}_{\bar{p}},\ \dot{z}(t)=-P^{-1}u(t),\ z(0)=x,\ z(t) \geq 0 \big\}.
		\end{aligned}
		\end{equation*}
	\end{definition}
	
	\subsection{$E$-$p$ transform}
	The main contribution of \cite{Evans2018} was presentation of the $E$-$p$ transform and results relating to it. For use in this paper, we reproduce the following definitions here:
	\begin{definition}
		Given a power reference $P^r\colon [0,+\infty)\to[0,+\infty)$, we define its E-p transform as the following function:
		\begin{equation}
		E_{P^r}(p)\doteq\int_0^\infty \textnormal{max}\big\{P^r(t)-p,0\big\}dt,
		\end{equation}
		interpretable as the energy required above any given power rating, $p$.
	\end{definition}
	\begin{definition}
		\label{definition: capacity}
		We define the \textbf{capacity} of a system to be the $E$-$p$ transform of the worst-case reference signal that it can meet, i.e. $\Omega_{\bar{p},x}(p)\doteq E_{R}(p)$, where $R(\cdot)$ is defined as
		\begin{equation}
		R(t)\doteq\sum_{i=1}^n\bar{p}_i[H(t)-H(t-x_i)],
		\end{equation}
		in which $H(\cdot)$ denotes the Heaviside step function.
	\end{definition}
	
	\section{Application to finite-duration interruption scenarios}
	\subsection{Explicit feedback policy}
	We present the following instantaneous feedback policy. Without loss of generality, reorder the states by descending value and group them into collections of equal value (leading to $q$ such groups),
		\begin{subequations}
			\label{eq: explicit policy}
			\begin{equation}
			\begin{aligned}
			x_1=...=x_{s_1} > ... > x_{s_{q-1}+1}=...=x_{s_q} .
			\end{aligned}
			\end{equation}
			
			\noindent Denoting by $\bar{U}_1,\bar{U}_2,...,\bar{U}_q$ stacked vectors of subset maximum powers, the explicit feedback law is then calculated as a fraction $r_i'$ of the maximum power $\bar{U}_i$ according to:
			\renewcommand{\arraystretch}{1.25}
			\begin{gather}
			r_i=\left\{
			\begin{array}{@{}ll@{}}
			1, & \text{if}\ \sum_{j\leq i}1^T\bar{U}_j\leq P^r \\
			0, & \text{if}\ \sum_{j< i}1^T\bar{U}_j\geq P^r \\
			\frac{P^r-\sum_{j< i}1^T\bar{U}_j}{1^T\bar{U}_i}, & \text{otherwise},
			\end{array}\right. \\
			r_i'=r_i\cdot\mathds{1}[x_{s_i}>0], \label{eq:indic}\\
			u^*(x,P^r)\doteq\begin{bmatrix}
			r_1'\bar{U}_1^T\ \dots\ r_q'\bar{U}_q^T
			\end{bmatrix}^T
			\end{gather}
		\end{subequations}
		\renewcommand{\arraystretch}{1}
		\hspace{-2.2mm} in which $\mathds{1}[\cdot]$ denotes the indicator function. Denoting by $z^*(\cdot)$ the state trajectory under the application of \eqref{eq: explicit policy}, the closed-loop dynamics are then
		\begin{equation}
		\label{eq:closed-loop dynamics}
		\dot{z}^*(t)=-P^{-1}u^*\big(z^*(t),P^r(t)\big) .
		\end{equation}
		Note that the policy \eqref{eq: explicit policy} extends that presented in \cite{Evans2017} by the addition of \eqref{eq:indic}.
	
	\subsection{Energy-not-served}
	We are concerned with the minimisation of energy-not-served, which we define as follows:	
	\begin{definition}
		The \textbf{energy-not-served}, $\Gamma$, under a reference $P^r(\cdot)$ and the resulting control signal $u(\cdot)$, is the shortfall in total energy output,
		\begin{equation}
		\Gamma\big(1^Tu,P^r\big)\doteq\int_0^\infty\textnormal{max}\big\{P^r(t)-1^Tu(t),0\big\}dt.
		\end{equation}
	\end{definition}

	\noindent We assume that $\Gamma$ is finite, which corresponds to the realistic case that $P^r(\cdot)$ is pointwise finite and has a finite integral. We then define the minimum ENS as follows:
	\begin{definition}
		Given a reference $P^r(\cdot)$ and a fleet with maximum power vector $\bar{p}$ in state $x$, we denote by $\Gamma^*_{\bar{p},x}(P^r)$ the minimum energy-not-served that might \textbf{feasibly} be achieved:
		\begin{equation}
		\Gamma^*_{\bar{p},x}(P^r)\doteq\min\big\{ \Gamma(1^Tu,P^r)\colon 1^Tu(\cdot)\in\mathcal{F}_{\bar{p},x}\big\}.
		\label{Gamma start def}
		\end{equation}
	\end{definition}

	\section{Results on Energy Optimality Under Loss of Supply}
	\subsection{Energy optimality}
	Having chosen as our objective the minimisation of energy-not-served, we now show that the policy \eqref{eq: explicit policy} achieves this aim. For clarity of argument the proof of Theorem~\ref{Theorem V.1} can be found in the \nameref{sec:appendix}.
	\begin{theorem}
		\label{Theorem V.1}
		Given any reference $P^r(\cdot)$, the policy \eqref{eq: explicit policy} minimises energy-not-served, i.e.
		\begin{equation}
		\Gamma(1^Tu^*,P^r)\leq\Gamma(1^T\tilde{u},P^r),
		\end{equation}
		in which $u^*(\cdot)$ and $\tilde{u}(\cdot)$ denote the outputs under the policy \eqref{eq: explicit policy} and any other choice respectively.
	\end{theorem}
	\begin{corollary}
		Given any reference $P^r(\cdot)$, the policy \eqref{eq: explicit policy} results in the minimum energy-not-served, $\Gamma^*_{\bar{p},x}(P^r)$.
	\end{corollary}
	\begin{corollary}
		Given any reference $P^r(\cdot)$, the policy \eqref{eq: explicit policy} results in a greedy optimisation of the energy-not-served, i.e.
		\begin{equation}
		\Gamma(1^Tu^*,P^r_{[0,T)})\leq\Gamma(1^T\tilde{u},P^r_{[0,T)})\ \ \forall T\geq0.
		\end{equation}
	\end{corollary}

	As a consequence of Theorem~\ref{Theorem V.1}, our policy should unambiguously be utilised whenever minimisation of energy-not-served is the objective of a grid operator or aggregator, regardless of the risk appetite of this entity (under the assumptions discussed above). In particular, this applies to system studies into the contribution of energy storage to capacity adequacy.
	
	\subsection{Capability assessments in the $E$-$p$ space}
	\label{sec: cap assess}
	As in \cite{Evans2018}, we here use our results on optimality to make fleet capability assessments in the $E$-$p$ space; this time in terms of energy-not-served. We define the \textit{max energy gap} and find a signal that has an energy-not-served equal to this value as follows. 
	Note that in the interest of readability the proof of Lemma~\ref{Lemma V.5} can be found in the \nameref{sec:appendix}.
	
	\begin{definition}
		Given a fleet with maximum power vector $\bar{p}$ in initial state $x$, we define the \textbf{max energy gap} of a signal $P^r(\cdot)$ to be the largest domination of the capacity curve by its $E$-$p$ transform, i.e.
		\begin{equation}
		\Phi_{\bar{p},x}(P^r)\doteq\sup_{p\geq0}\big[\max\{E_{P^r}(p)-\Omega_{\bar{p},x}(p),0\}\big].
		\end{equation}
	\end{definition}

	\noindent Given a reference $P^r(\cdot)$, we define a capped signal with level $\tilde{p}$ as $\bar{P}^r(\cdot)$, constructed according to
	\begin{equation}
	\bar{P}^r(t)=\min\{P^r(t),\ \tilde{p}\}.
	\label{eq: capped signal}
	\end{equation}
	
	\begin{lemma}
		\label{Lemma V.5}
		The capped signal $\bar{P}^r(\cdot)$ defined by $P^r(\cdot)$ fulfilling the following condition at $\tilde{p}$:
			\begin{equation}
			E_{P^r}(\tilde{p})=\Phi_{\bar{p},x}(P^r)
			\label{eq: p_tilde}
			\end{equation}
		results in an energy-not-served equal to the max energy gap, i.e. $\Gamma(\bar{P}^r,P^r)=\Phi_{\bar{p},x}(P^r)$. Moreover, this signal has an $E$-$p$ transform equal to
		\begin{equation}
		E_{\bar{P}^r}(p)=\max\{E_{P^r}(p)-\Phi_{\bar{p},x}(P^r),0\}.
		\end{equation}
	\end{lemma}
	
	\noindent We are then able to derive the following result, the proof of which can also be found in the \nameref{sec:appendix}:
	\begin{theorem}
		\label{Theorem V.6}
		For any given reference $P^r(\cdot)$, the max energy gap is equal to the minimum energy-not-served, i.e. $\Gamma_{\bar{p},x}^*(P^r)=\Phi_{\bar{p},x}(P^r)$.
	\end{theorem}
	\noindent Combination of Lemma~\ref{Lemma V.5} and Theorems~\ref{Theorem V.1} and \ref{Theorem V.6} then leads to the following Corollaries:
	\begin{corollary}
		The max energy gap $\Phi_{\bar{p},x}(P^r)$ is equal to the energy-not-served under the policy \eqref{eq: explicit policy}.
	\end{corollary}
	\begin{corollary}
		\label{cor: peak shaving}
		The saturation signal $\bar{P}^r(\cdot)$ of level $\tilde{p}$ defined as in \eqref{eq: p_tilde} results in the minimum energy-not-served, i.e.
		\begin{equation}
		\Gamma(\bar{P}^r,P^r)=\Gamma_{\bar{p},x}^*(P^r).
		\end{equation}
	\end{corollary}

	\noindent Making use of Corollary~\ref{cor: peak shaving}, we are then able to form an alternative policy that is energy-optimal: the \emph{peak-shaving} policy allocated as follows. Find $\tilde{p}$ satisfying \eqref{eq: p_tilde}, then compose the capped signal $\bar{P}^r(\cdot)$ according to \eqref{eq: capped signal} and implement the policy \eqref{eq: explicit policy} to meet $\bar{P}^r(\cdot)$. Note, however, that this policy is non-causal; it requires perfect foresight of the reference in order to determine the saturation level $\tilde{p}$. As a result it is in general only energy optimal with respect to the full reference signal, unlike the policy \eqref{eq: explicit policy} which is energy optimal in a greedy sense. The cumulative energy delivered under the peak-shaving policy can therefore be seen to ``catch up" with the greedy optimal output as time progresses.
	
	\section{Discrete time policy}
	\subsection{Piecewise constant policy}
	Thus far the proposed policy has been posed in continuous time, however we envision that discrete settings would be more relevant in practice. Examples of such settings might include market clearing or dispatch with fixed intervals, or simulations with time steps. We consider piecewise constant signals and assume that there is a discrete clock indexed by $t_i$; often such sample instants will be equidistant but they need not be. We firstly focus on a single time interval $[t_1,t_2)$. Given a reference that is constant across this time interval, we then find the final state under the policy \eqref{eq: explicit policy} and construct a constant input that reaches it in the same time as follows. For clarity of argument the proof of Lemmas~\ref{Lemma VI.1} and \ref{Lemma VI.2} can be found in the \nameref{sec:appendix}.
	\begin{lemma}
		\label{Lemma VI.1}
		Given a starting state $x$ at time $t_1$ and a reference signal $P^r(\cdot)$ that is constant across the time interval $[t_1,t_2)$, i.e. such that $P^r(\tau)=P^r\ \forall \tau\in[t_1,t_2)$, the state resulting from the implementation of the policy \eqref{eq: explicit policy} across this interval, $z\doteq z^*(t_2)$, can be found as follows:
		\begin{equation}
		z_i=\min\big\{\max\big\{x_i-\Delta t,\hat{z}\big\},x_i\big\},
		\label{eq:x_star_calc2}
		\end{equation}
		in which
		\begin{equation}
		\hat{z}\doteq\inf\bigg\{\hat{x}\geq0\colon\sum_{i=1}^n\bar{p}_i\max\big\{\min\big\{x_i-\hat{x},\Delta t\big\},0\big\}\leq P^r\Delta t\bigg\}
		\label{eq:x_star_calc}
		\end{equation}
		and $\Delta t$ is the length of the interval, $\Delta t\doteq t_2-t_1$.	
	\end{lemma}
	\begin{lemma}
		\label{Lemma VI.2}
		Given an initial state $x\doteq x(t_1)$ and a reference signal $P^r(\cdot)$ such that $P^r(\tau)=P^r\ \forall \tau\in[t_1,t_2)$, the final state under the implementation of the policy \eqref{eq: explicit policy}, $z\doteq z^*(t_2)$, can be reached by a constant control signal allocated as
		\begin{equation}
		u=P\cdot\bigg(\frac{x-z}{t_2-t_1}\bigg)
		\label{eq:u allocation}
		\end{equation}
		in the same time. This signal is feasible across the interval.
	\end{lemma}

	\noindent Thus Lemmas~\ref{Lemma VI.1} and \ref{Lemma VI.2} allow us to form a constant input that returns the optimal final state given a constant reference. This is constructed according to
	\begin{equation}
	u_i=\bar{p}_i\cdot\textnormal{max}\bigg\{\textnormal{min}\bigg\{\frac{x_i-\hat{z}}{\Delta t},1\bigg\},0\bigg\},
	\label{eq:DTOP}
	\end{equation}
	in which $\Delta t\doteq t_2{-}t_1$ and $\hat{z}$ is allocated according to \eqref{eq:x_star_calc}. The piecewise constant policy is then simply formed by implementing \eqref{eq:DTOP} from each time instant to the next. This policy satisfies the following results:	
	\begin{theorem}
		\label{the discretePolicy}
		The piecewise constant policy results in a trajectory that is optimal at all $t_i$.
	\end{theorem}
	\begin{proof}
			Follows directly from Theorem~\ref{Theorem V.1} and Lemmas~\ref{Lemma VI.1} and \ref{Lemma VI.2}.
	\end{proof}
	\begin{corollary}
		In a discrete time system the piecewise constant policy is optimal.
	\end{corollary}

	 It should be noted that the authors of \cite{Edwards2017} considered the task of directly assigning a piecewise constant input according to \eqref{eq: explicit policy}, but showed that this is suboptimal; \eqref{eq:DTOP} is in fact the correct way to convert \eqref{eq: explicit policy} into discrete time. For the special case of a single store, the authors of \cite{Edwards2017} did propose a greedy policy that is equivalent to \eqref{eq:DTOP}.

	\subsection{Discrete time algorithm}
	Figure~\ref{fig:alg} describes a discrete time algorithm that implements the policy \eqref{eq:DTOP} as follows. We firstly find $\hat{z}$ according to \eqref{eq:x_star_calc}, as is demonstrated in Figure~\ref{fig:alg_explainer}. To this end, we consider the total energy output as a function of candidate $\hat{z}$ values (right panel, transposed). A descending list of potential discontinuities is constructed ($y_i$), which arise because each device $D_j$ contributes partially if $\max\{x_j-\Delta t,\ 0\}<\hat{z}< x_j$, with the energy output decreasing linearly with $\hat{z}$ over this interval, or maximally if $\hat{z}\leq x_j -\Delta t$. We then iterate over this list to find the interval $[y_{i-1},\ y_i]$ that contains $\hat{z}$ fulfilling \eqref{eq:x_star_calc} and, if necessary, linearly interpolate between its limits. Having found $\hat{z}$, we then allocate $u$ according to \eqref{eq:DTOP}.

	\addtocounter{figure}{1}
	\begin{figure}[h]
		\begin{algorithm}[H]
			\caption{\hspace{5mm}$\bm{u}=$ Constant input policy$(P^r,\bm{x},\bar{\bm{p}}, \Delta t)$}
			\begin{algorithmic}[1]
				\setstretch{1.3}\State $\bm{y}\gets\text{unique}\begin{bmatrix}
				\text{sort-descending}\begin{bmatrix}
				\bm{x}\\
				\max\{\bm{x}-\Delta t\bm{1},\bm{0}\}
				\end{bmatrix}\end{bmatrix}$
				\State $\bar{E}\gets0$ \hspace{22.5mm}$\triangleright\ $\begin{footnotesize}\textit{Upper bound on energy in this iteration}\end{footnotesize}
				\State $i\gets0$ \hspace{61mm}$\triangleright\ $\begin{footnotesize}\textit{Counter}\end{footnotesize}
				\Do
				\State $i\gets i+1$
				\State $\underaccent{\bar}{E}\gets\bar{E}$ \hspace{16.5mm}$\triangleright\ $\begin{footnotesize}\textit{Lower bound on energy in this iteration}\end{footnotesize}
				\State $\bar{E}\gets\bar{\bm{p}}^T\max\{\min\{\bm{x}-y_i\bm{1},\Delta t\bm{1}\},\bm{0}\}$
				\Until {$\bar{E}\geq P^r\Delta t$ \textbf{or} $i=\text{length}[\bm{y}]$}
				\If {$\bar{E}\leq P^r\Delta t$}\State $\hat{z}\gets y_i$ \hspace{14mm}$\triangleright\ $\parbox[t]{.55\columnwidth}{\begin{footnotesize}\textit{Upper bound equals energy requirement, or $P^r$ infeasible}\end{footnotesize}}
				\Else 
				\State $\hat{z}\gets y_{i-1}+\frac{P^r\Delta t-\underaccent{\bar}{E}}{\bar{E}-\underaccent{\bar}{E}}\big(y_i-y_{i-1}\big)$ \hspace{11mm}{$\triangleright\ $\begin{footnotesize}\textit{Interpolation}\end{footnotesize}}
				\EndIf
				\State $\bm{u}\gets\bar{\bm{p}}\circ\textnormal{max}\big\{\textnormal{min}\big\{\frac{\bm{x}-\hat{z}\bm{1}}{\Delta t},\bm{1}\big\},0\big\}$
			\end{algorithmic}
		\label{alg: DTOP}
		\end{algorithm}
		\vspace{-3mm}
		\begin{subfigure}[l]{\columnwidth}
		\caption{	Constant input policy algorithm. Bold font is used for vectors (all of length $n$) and normal font for scalars. $\bm{a}\circ\bm{b}$ denotes the Hadamard product of $\bm{a}$ and $\bm{b}$.}
		\label{fig:discrete_alg}
		\end{subfigure}
		\quad
		\begin{subfigure}[l]{\columnwidth}
			\includegraphics[width=0.99\columnwidth]{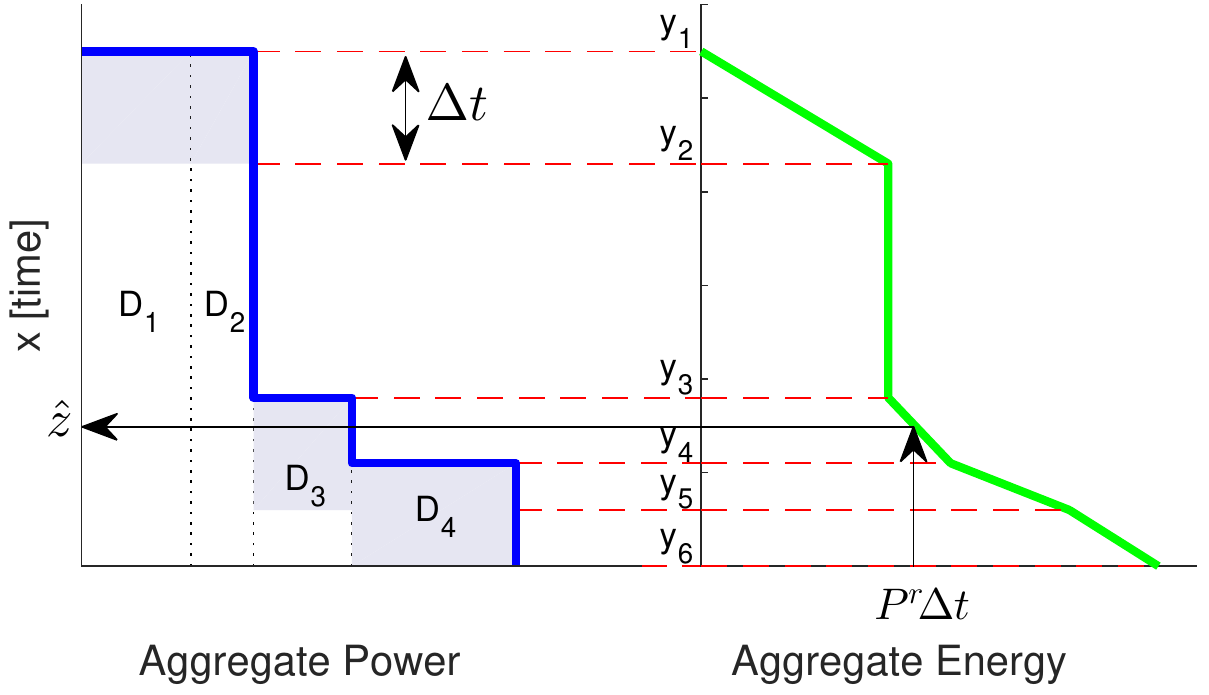}
		    \vspace{-1.5mm}
			\caption{Method for finding $\hat{z}$ using the energy balance of \eqref{eq:x_star_calc}. The filled blue areas represent the stored energy that can be accessed in $\Delta t$.}
			\label{fig:alg_explainer}
		\end{subfigure}
		\addtocounter{figure}{-1}
		\caption{Algorithm for constructing a constant input policy. }
		\label{fig:alg}
	\end{figure}

	\section{Numerical Results}
	\subsection{Simple demonstration case}
	We firstly demonstrate the operation of the discrete time policy in a simple case study. We choose as our fleet 4 devices of the following ratings (energy, power): (8 kWh, 2 kW), (12~kWh, 4 kW), (6 kWh, 3 kW) and (7 kWh, 7 kW). We then determine the shortfall when this fleet is requested to provide a 4 h profile consisting of [4, 18, 12, 1] kW for 1 h each. Table~\ref{tab:example} presents a step-by-step breakdown of the state of charge and power delivered by each device. A graphical comparison of power delivered and power requested is shown in Fig.~\ref{fig:simple_example}, in a direct representation and via an $E$-$p$ analysis. The latter highlights how the max energy gap in $E$-$p$ space is equal to the observed energy-not-served under the optimal policy; both are equal to the minimum energy-not-served.
	
	\renewcommand{\arraystretch}{1.1}
	\begin{table}[h]
	\centering
	\begin{tabular}{|cc|c|cccc|}
		\hline
		\multirow{2}{*}{variable} & \multirow{2}{*}{unit} & \multirow{2}{*}{limit} & \multicolumn{4}{|c|}{time step}\\
		 &&&  1 & 2 & 3 & 4 \\
        \hline\hline
        $P^r$& [kW] && 4 & 18  & 12 & 1 \\
        \hline
		$x_1$& [h]  && 4 & 3  & 2  & 1 \\
		$x_2$& [h]  && 3 & 2.5 & 1.5 & 0.5 \\
		$x_3$& [h]  && 2 & 2  & 1 & 0 \\
		$x_4$& [h]  && 1 & 1  &0  & 0 \\
		\hline
		$\hat{z}$& [h]  && 2.5  & 0 & 0 & 0.5 \\
		\hline
		$u_1$& [kW]  & 2 & 2  & 2 & 2 & 1 \\
		$u_2$& [kW]  & 4 & 2  & 4 & 4 & 0 \\
		$u_3$& [kW]  & 3 & 0 & 3 & 3 & 0 \\
		$u_4$& [kW]  & 7 & 0 & 7 & 0 & 0 \\
		\hline
		ENS& [kWh]  && 0 & 2 & 3 & 0 \\
	\hline
	\end{tabular}
	\vspace{-1mm}
	\caption{\small Step-by-step analysis for the four-device example.}
	\label{tab:example}
    \end{table}
    \renewcommand{\arraystretch}{1}
	
	\vspace{-4.9mm}
	\begin{figure}[h]
		\begin{subfigure}[t]{0.45\columnwidth}
		    \centering
			\includegraphics[width=\columnwidth]{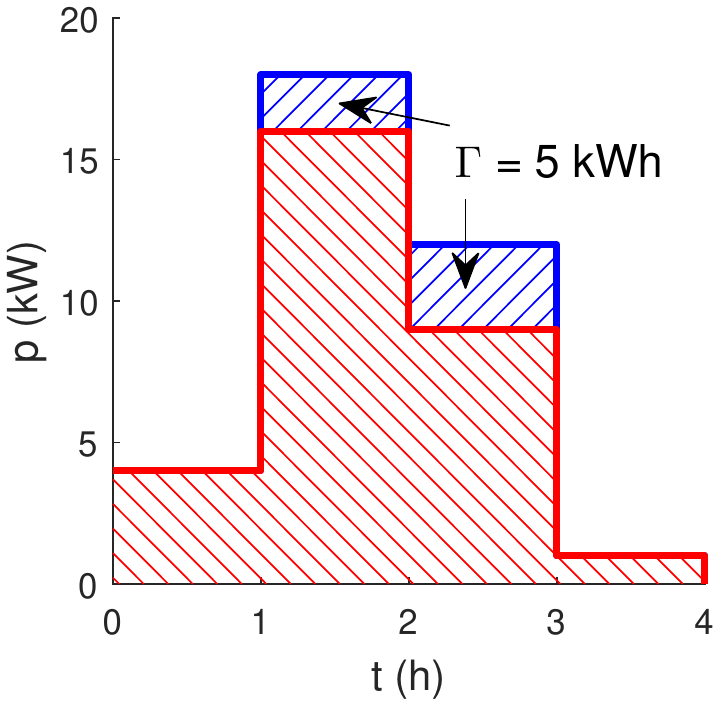}
			\caption{Inability to meet the requested demand. The cumulative output is shown in hatched red, with unserved demand in hatched blue.}
			\label{fig:simple_time_series}
		\end{subfigure}
		~
		\begin{subfigure}[t]{0.45\columnwidth}
		    \centering
			\includegraphics[width=\columnwidth]{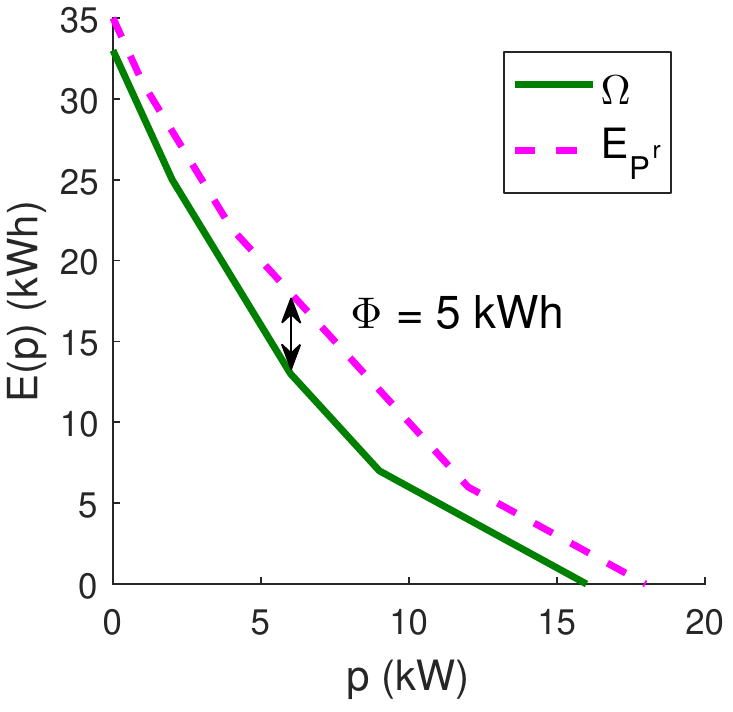}
			\caption{$E$-$p$ transform of the request, indicating the max energy gap.}
			\label{fig:simple_Ep}
		\end{subfigure}
		\caption{The results of the four-device example.}
		\label{fig:simple_example}
	\end{figure}
	
	\vspace{-3mm}
	\subsection{GB capacity adequacy case study}
	\subsubsection{Recharging policy}
	\label{sec:recharge}
	In order to perform a system adequacy-inspired case study, we here consider the recharging of the storage units between loss of supply events. We denote by $\bar{e}_i$ and $\ubar{p}_i$ the maximum energy and charging rate of device $D_i$, respectively, and  define its maximum time-to-go value as $\bar{x}_i\doteq\bar{e}_i/\bar{p}_i$. For notational consistency, the maximum charging rate $\ubar{p}_i$ and $P^r$ (while charging) have a negative sign. We then assume that all devices have an identical combined charging and discharging efficiency, $\eta$, and that the objective of the aggregator while charging is to greedily maximise the feasible set (of discharge signals), with no cross-charging between devices. From Theorem~\ref{Theorem V.6}, we know that the max energy gap of a given reference is equal to the ENS under that signal, and so during recharging operation the aggregator should aim to increase the $E$-$p$ curve as much as possible. The best recharge policy, therefore, is for devices to be filled starting from the device with the smallest time-to-go value, since in this way the $E$-$p$ curve will be pulled up at all power levels. This can be implemented via an inversion of \eqref{eq:DTOP} as follows. The final state $z_i$ of device $D_i$ is constrained both by its maximum capacity, $\bar{x}_i$, and the maximum charge increment $( -\eta\ubar{p}_i\Delta t/\bar{p}_i)$. We combine these constraints to give
	\begin{equation}
	z_i\leq\bar{z}_i\doteq\min\bigg\{x_i-\frac{\eta\ubar{p}_i\Delta t}{\bar{p}_i},\ \bar{x}_i\bigg\}.
	\label{eq:zbar}
	\end{equation}
	The recharge policy is then allocated according to:
	\begin{equation}
	u_i=-\frac{\bar{p}_i}{\eta\Delta t}\cdot\max\{\min\{\hat{z},\bar{z}_i\}-x_i,0\},
	\label{eq: explicit policy recharge}
	\end{equation}
	in which
	\begin{equation}
	\hat{z}\doteq\sup\bigg\{\hat{x}\leq z'\colon\sum_{i=1}^n\bar{p}_i\max\{\min\{\hat{x},\bar{z}_i\}-x_i,0\big\}\leq -P^r\Delta t\bigg\},
	\label{eq:E_balance_charge}
	\end{equation}
	$z'\doteq\max\{\bar{z}_i\colon D_i\in\mathcal{N}\}$, $\Delta t\doteq t_2{-}t_1$ and $\bar{z}$ is defined as in \eqref{eq:zbar}. Note that this policy can be implemented via the following amendments to Algorithm~\ref{alg: DTOP}. The elements of $\bm{y}$ (line 1) should be found as $y_j\in\{x_i,\bar{z}_i\colon D_i\in\mathcal{N}\}$ and considered in ascending order. Each upper bound on energy (line 7) should then be found using \eqref{eq:E_balance_charge} and compared to $-P^r\Delta t$, including in the interpolation (line 12). Finally, the input (line 14) should be allocated according to \eqref{eq: explicit policy recharge}.
	
	Unlike the discharge policy, we do not claim that this recharge policy is optimal in a set-theoretic sense. However, we will investigate whether devices are effectively fully refilled between discharge events according to implementation of this policy to our case study; if this is indeed the case then we are able to neglect any effects which may occur as a result of this potential sub-optimality.
	
	\subsubsection{Alternative policy choices}
	To demonstrate the advantage of implementing the proposed policy, we compare it to the following alternatives (details can be found in the references): \textit{Lowest Power First} \cite{Evans2017,Evans2018}, \textit{Proportion of Power} \cite{Evans2017,Evans2018} and \textit{Proportional Discharge} \cite{Edwards2017}. To adapt the first two for discrete time settings, we upper-bound the dispatch of device $D_i$ at each sample instant by its interval-limited maximum power, $\bar{p}_i\min\{x_i/\Delta t,1\}$, and directly allocate a piecewise constant input according to the feedback law. For consistent results we apply the recharge policy \eqref{eq: explicit policy recharge} regardless of the discharge policy chosen. We also investigate the baseline case in which no storage is contracted.
	
	\subsubsection{System model}
	These policy choices were each applied to the following GB case study system. Annual demand and wind output traces were generated by sampling each independently from a set of annual traces. Historical GB demand measurements for 2006-2015 were used (net demand, excluding exports and recharging of storage units, and corrected for (estimated) output from embedded renewable generation; data kindly provided by Iain Staffell \cite{Staffell2017}). GB wind power output for the period 1985-2014 was synthesised for an assumed 10~GW installed capacity and a capacity factor time series that was derived from MERRA reanalysis data for wind speeds and an assumed constant distribution of wind generation sites \cite{Staffell2016}. The conventional generation portfolio consisted of 63~GW installed capacity, distributed as: 1200~MW (20 units), 600~MW (40), 250~MW (40), 120~MW (20), 60~MW (20), 20~MW (40), 10~MW (60). Annual traces for conventional generating capacity were generated by assuming independence between units, an availability of 0.9 (forced outage rate of 0.1), a discrete time Poisson process for failures and repairs (constant failure and repair rates) and a mean time between subsequent failure events of 2000 h. The system as initialised had an LOLE of 2.9 h/y (computed by convolution). To this system was added the storage that was contracted in the GB 2018 T-4 capacity auction and reported in \cite{NG}, assuming that each of the 27 contract-winning bids was implemented via a single device of unity efficiency, with each power rating multiplied by 3 to demonstrate the effect of higher storage proliferation. Duration ratings were inferred from the listed de-rating factors and Table E1 of the same text.
	
	\subsubsection{Results}	
	The observed EENS is shown in Table~\ref{tab:observed ENS}, in which it can be seen that the optimal policy resulted in the lowest value out of the options considered, as well as the joint-lowest LOLE in this example. We point out here that, despite differences that are smaller than the confidence intervals, these estimates can be directly compared because they were computed from the same set of generation margin traces. In our analysis we consider each contiguous period of supply requests received by the storage units to be a shortfall event. Because most shortfalls that occurred were of a relatively large magnitude as compared to the available storage, any of the other policies was able to closely approximate the lower-bound on EENS achieved by the optimal policy, and resulted in a large improvement over the no storage base-case. Example reference traces and the aggregate output supplied under the optimal policy to meet such requests can be seen in the left-hand plots of Figure~\ref{fig:trace examples}. The right-hand plots of the same figure then show the range of $E$-$p$ curves corresponding to these examples, and demonstrate direct calculation of ENS via an $E$-$p$ curve. Note that infeasibilty can occur in different ways, including when the total energy or maximum power of the request is above the total fleet rating, or both. In addition, it is possible for a request to be feasible in terms of power and energy in isolation, yet infeasible due to limitations because of heterogeneity of the fleet. See \cite{Evans2018} for further discussion of this point. The fleet was found to be fully recharged by the start of 99.4\% of observed shortfall events, hence we conclude that our choice of recharging policy did not significantly affect the results.
	
	\addtolength{\textfloatsep}{-1mm}
	\begin{table}[h]
		\centering
		\begin{tabular}{|l|c|c|c|c|}
			\hline
			Policy & LOLE (h/y) & EENS (MWh/y) \\
			\hline
			Optimal Policy & 1.74 $\pm$ 0.09 & 2431 $\pm$ 165 \\
			Lowest Power First & 1.74 $\pm$ 0.09 & 2443 $\pm$ 165 \\
			Proportion of Power & 1.74 $\pm$ 0.09 & 2435 $\pm$ 165 \\
			Proportional Discharge & 1.85 $\pm$ 0.09 & 2438 $\pm$ 165 \\
			No Storage & 2.98 $\pm$ 0.12 & 3810 $\pm$ 208 \\
			\hline
		\end{tabular}
		\vspace{-1mm}
		\caption{\small Observed LOLE and EENS, across the policy options. 95\% confidence intervals are reported, as estimated from the 10,000 Monte Carlo sampled years.}
		\label{tab:observed ENS}
	\end{table}
	
	\begin{figure}[h]
		\captionsetup[subfigure]{aboveskip=-0.5mm,belowskip=-0.5mm}
		\centering
		\begin{subfigure}[l]{\columnwidth}
			\includegraphics[width=\columnwidth]{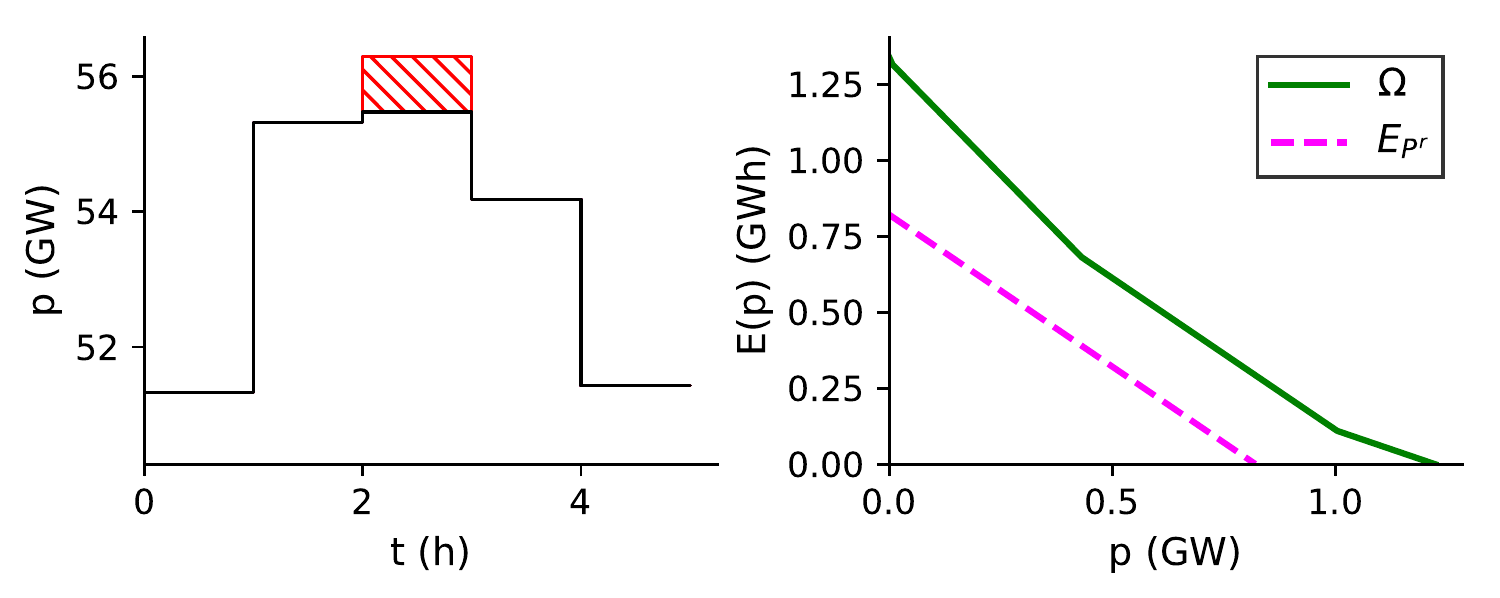}
			\vspace{-5mm}
			\caption{A feasible request}
			\label{fig:Feasible}
		\end{subfigure}
		\quad
		\begin{subfigure}[l]{\columnwidth}
			\includegraphics[width=\columnwidth]{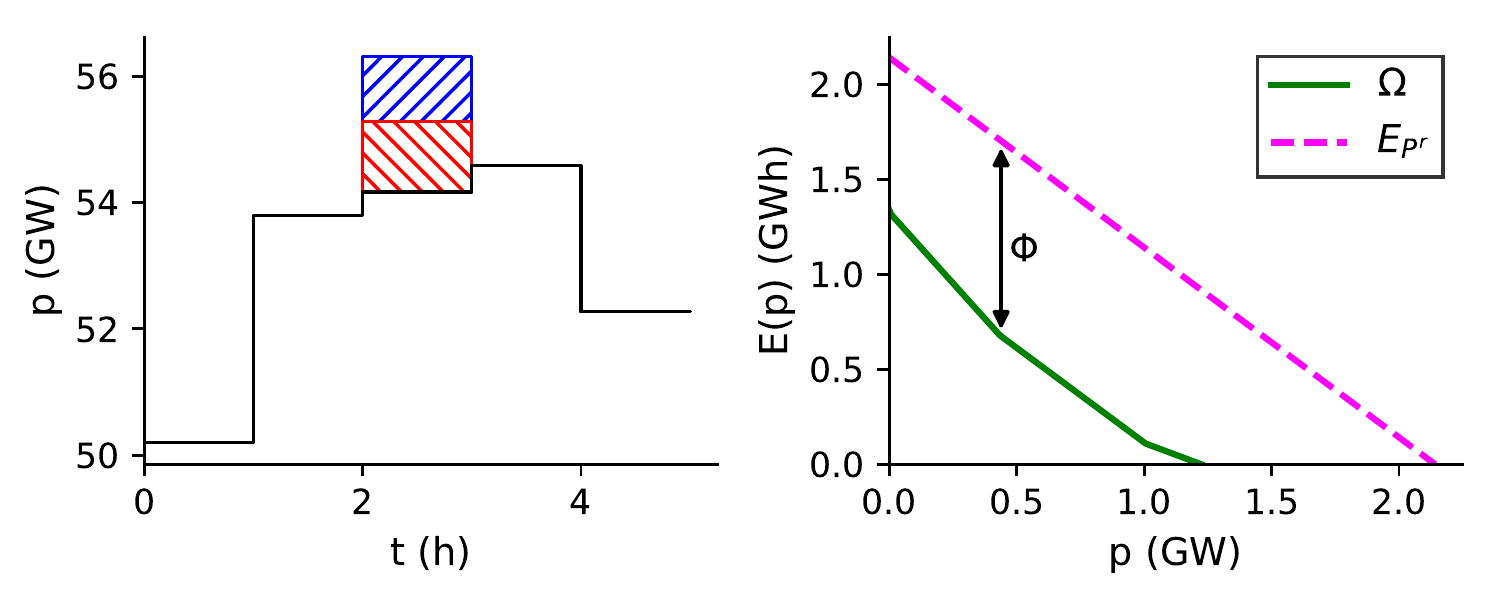}
			\caption{A request that is infeasible in terms of both power and energy}
			\label{fig:Small}
		\end{subfigure}
		\quad
		\begin{subfigure}[l]{\columnwidth}
			\includegraphics[width=\columnwidth]{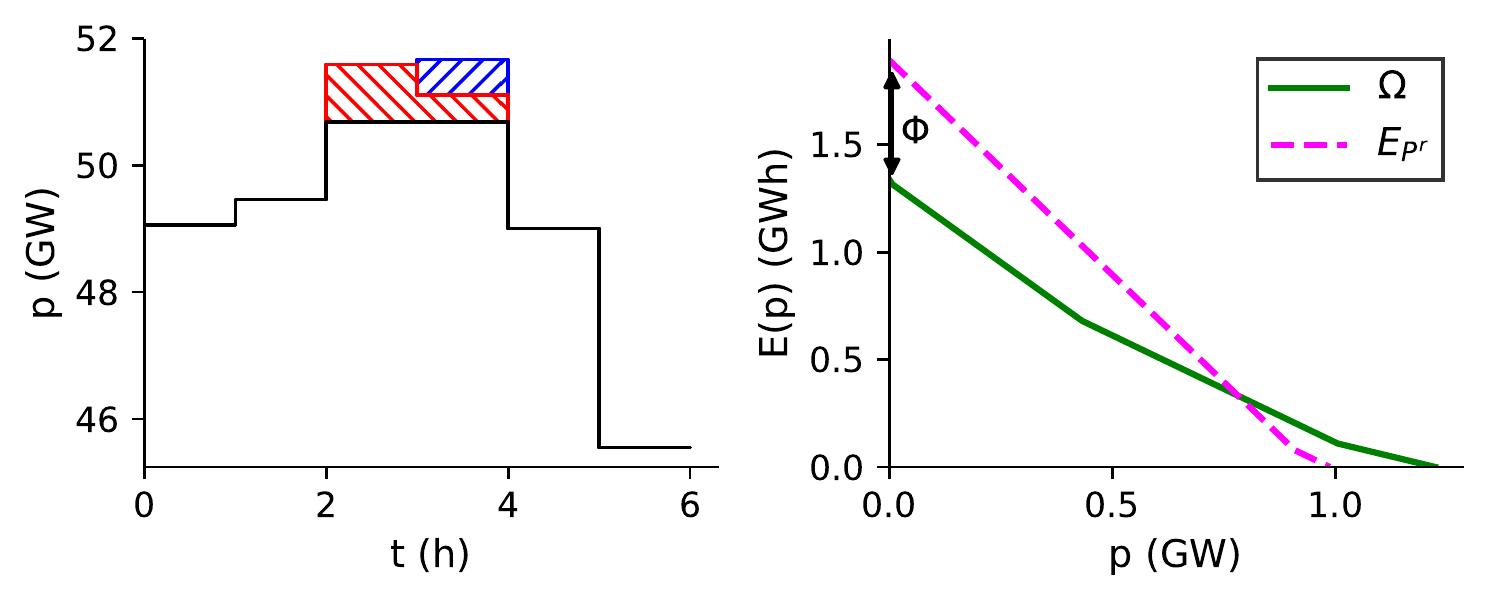}
			\caption{A request that is infeasible in terms of energy alone}
			\label{fig:Energy}
		\end{subfigure}
		\quad
		\begin{subfigure}[l]{\columnwidth}
			\includegraphics[width=\columnwidth]{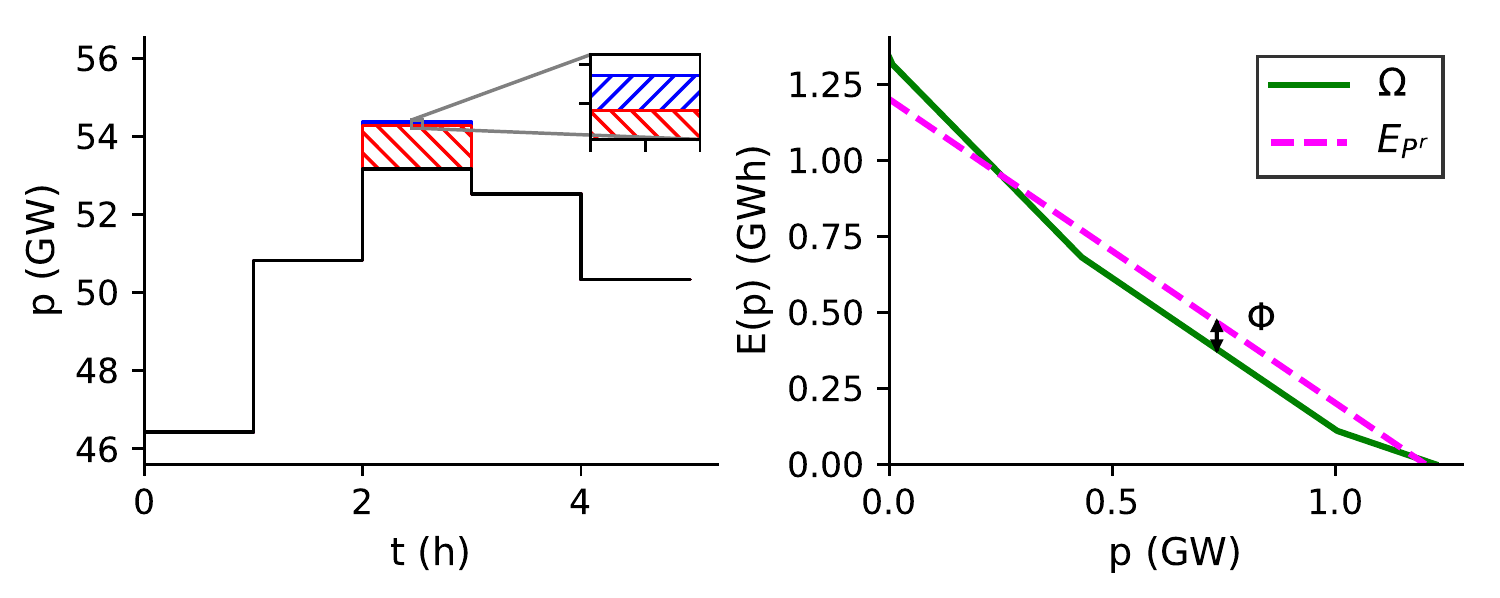}
			\caption{A request that is infeasible as a result of device heterogeneity, but not in terms of power or energy in isolation}
			\label{fig:Neither}
		\end{subfigure}
		\vskip 1.5mm
		\caption{Example types of request infeasibility observed. Left panels shown cumulative power traces (black: generator-supplied demand; hatched red: storage-supplied demand; hatched blue: unserved demand). Right panels show the E-p representation (solid green: capacity; dashed magenta: request) as well as the max energy gap (for infeasible traces).}
		\label{fig:trace examples}
	\end{figure}
	
	\subsubsection{Discussion}
	The results obtained have practical relevance for the GB capacity market. The TSO, National Grid, currently establishes capacity payments for energy storage providers by computing the EENS-determined capacity value of reference units under the assumption of perfect foresight, using 4 different dispatch algorithms, each of which leads to slightly different results \cite{NG}. The dispatch algorithm proposed in this paper, by contrast, achieves a minimisation of EENS in a greedy manner, thereby circumventing the reliance on perfect foresight (or explicit forecasts). We note that, in reality, storage units will most likely not be dispatched in a fully optimal manner, but an assumption of optimality is typically embedded in long-term adequacy studies, e.g. by ignoring unit commitment problems for conventional generation. In this light, our greedy dispatch algorithm is a suitable default model for dispatch of storage in adequacy studies of the GB system, and of other systems where a minimisation of EENS is desired. 
	
	\section{Conclusions and Future Work}
	
	This paper has considered the optimal dispatch of energy-constrained heterogeneous storage units to maximise security of supply. We have presented a greedy policy that minimises unserved energy, regardless of system demand. We have then discussed how analysis in the $E$-$p$ space can be used to determine capacity adequacy of fleets of storage devices, including the immediate determination of energy-not-served under a received reference. We have finally provided an algorithm for implementation of the optimal policy in discrete time settings. The algorithm can be used operationally by aggregators of heterogeneous storage fleets. Moreover, a case study has demonstrated its suitability for EENS-minimising dispatch of storage units in system adequacy studies.
	
	In future work the authors plan to extend these results to include the accommodation of cross-charging between devices and network constraints.

	\section*{Appendix}
	\label{sec:appendix}
	\addcontentsline{toc}{section}{Appendix}
	\addtocounter{section}{1}
	\renewcommand{\theequation}{A.\arabic{equation}}	
	\setcounter{equation}{0}
	\vspace{-1mm}
	\subsection{Proof of \Cref{Theorem V.1}}
	\noindent As an interim procedure, we utilise the following framework:
	\begin{enumerate}
		\item We permit the controller to virtually draw power from depleted devices.
		\item We introduce a virtual store, $D_{n+1}$, that is initially empty and has no power limit.
		\item We maintain the condition that the total (real and virtual) power output must meet the reference at all time instants
	\end{enumerate}
	We then construct the following feedback policy:
		\begin{subequations}
		\label{eq: orig explicit policy}
		\begin{gather}
		\kappa\doteq\begin{bmatrix}
		r_1\bar{U}_1^T\ \dots\ r_q\bar{U}_q^T\
		\end{bmatrix}^T\\
		u^*\doteq\begin{bmatrix}
		\kappa^T\ \ (P^r{-}1^T\kappa)
		\end{bmatrix}^T,
		\end{gather}
		\end{subequations}
		in which $r_i$ is allocated according to \eqref{eq: explicit policy}.
		The actual output, neglecting virtual requests, under the policy \eqref{eq: orig explicit policy} matches that of the policy \eqref{eq: explicit policy}. Hence, Theorem~\ref{Theorem V.1} follows trivially from the following Lemma:
			
		\begin{lemma}
			\label{lem e_optimality1}
			Given any reference $P^r(\cdot)$, the policy \eqref{eq: orig explicit policy} minimises virtual requests, i.e.
			\begin{equation}
			\sum_{i=1}^{n+1}\bar{p}_i\min\{z^*_i(\infty),0\}\geq\sum_{i=1}^{n+1}\bar{p}_i\min\{\tilde{z}_i(\infty),0\},
			\end{equation}
			in which $z^*(\cdot)$ and $\tilde{z}(\cdot)$ denote the trajectories under the policy \eqref{eq: orig explicit policy} and any other choice respectively.
		\end{lemma}		
		\begin{proof}
		Consider grouping the devices into exactly two sets as follows:
		\begin{subequations}
			\begin{gather}
			\mathcal{Q}\doteq\{D_i\colon z^*_i(\infty)\leq0\},\\
			S\doteq\mathcal{N}\setminus\mathcal{Q},
			\end{gather}
		\end{subequations}
		those that will be depleted by the final time under the policy \eqref{eq: orig explicit policy} and those that will not, respectively. Note that $D_{n+1}{\in}\mathcal{Q}$ as it has an initial state value of 0. In general, we know that the policy \eqref{eq: orig explicit policy} chooses to run the devices contained in $\mathcal{S}$ as much as possible, since it allocates starting with devices of highest time-to-go value yet depletes these devices last. This remains true with the inclusion of the virtual device, since the policy \eqref{eq: orig explicit policy} only calls upon this device when the reference exceeds the aggregate rating of all other devices. Utilising the notation that $z_{\mathcal{S}}$ denotes the $z$-vector truncated to the devices which are elements of $\mathcal{S}$, and likewise that $1_{\mathcal{S}}$ and $P_{\mathcal{S}}$ denote the unity vector and $P\text{-matrix}$ truncated to the corresponding elements, we are able therefore to say that
		\begin{equation}
		1_{\mathcal{S}}^TP_{\mathcal{S}}\dot{z}^*_{\mathcal{S}}(t) \leq	1_{\mathcal{S}}^TP_{\mathcal{S}}\dot{\tilde{z}}_{\mathcal{S}}(t)\ \ \forall t\geq0.
		\label{eq: S set}
		\end{equation}
		This condition, coupled with the requirement that the aggregate output must match the reference, i.e.
		\begin{equation}
		1_{\mathcal{S}}^TP_{\mathcal{S}}\dot{z}^*_{\mathcal{S}}(t)+1_{\mathcal{Q}}^TP_{\mathcal{Q}}\dot{z}^*_{\mathcal{Q}}(t)=1_{\mathcal{S}}^TP_{\mathcal{S}}\dot{\tilde{z}}_{\mathcal{S}}(t)+1_{\mathcal{Q}}^TP_{\mathcal{Q}}\dot{\tilde{z}}_{\mathcal{Q}}(t)=P^r(t),
		\end{equation}
		leads to
		\begin{equation}
		1_{\mathcal{Q}}^TP_{\mathcal{Q}}\dot{z}^*_{\mathcal{Q}}(t) \geq 	1_{\mathcal{Q}}^TP_{\mathcal{Q}}\dot{\tilde{z}}_{\mathcal{Q}}(t)\ \ \forall t\geq0.
		\label{eq:Qcond}
		\end{equation}
		Integration of \eqref{eq:Qcond} gives
		$1_{\mathcal{Q}}^TP_{\mathcal{Q}}z^*_{\mathcal{Q}}(\infty)\geq1_{\mathcal{Q}}^TP_{\mathcal{Q}}\tilde{z}_{\mathcal{Q}}(\infty)$. Hence,
		\begin{equation}
		\begin{aligned}
		\sum_{i=1}^{n+1}\bar{p}_i\min\{z^*_i(\infty),0\}&=\sum_{D_i\in\mathcal{Q}}\bar{p}_i\min\{z^*_i(\infty),0\}\\
		&\geq\sum_{D_i\in\mathcal{Q}}\bar{p}_i\min\{\tilde{z}_i(\infty),0\}\\
		&\geq\sum_{i=1}^{n+1}\bar{p}_i\min\{\tilde{z}_i(\infty),0\}. 
		\end{aligned}
		\vspace{-5mm}
		\end{equation}
		\end{proof}
		
		\vspace{-6mm}
		\subsection{Proof of \Cref{Lemma V.5}}
			\noindent Firstly, note that such a value of $\tilde{p}$ must both exist and be unique as a result of the convex nature of the $E$-$p$ transform \cite{Evans2018}. Defining $p^*$ as any power level satisfying
			\begin{equation}
			p^*\in \arg\sup_{p\geq0}\big[\max\{E_{P^r}(p)-\Omega_{\bar{p},x}(p),0\}\big],
			\label{eq:p star}
			\end{equation}
			and noting that $E_{P^r}(p)<E_{P^r}(\tilde{p})\ \ \forall p>\tilde{p}$, we are also able to say that $p^*\leq\tilde{p}$ for all $p^*$ as in \eqref{eq:p star}.
			
			Now, consider the region defined by $p\leq \tilde{p}$, and deduce that 
			\begin{equation}
			\max\{\min\{P^r(t),\tilde{p}\}-p,0\}=\left\{
			\begin{array}{@{}ll@{}}
			\tilde{p}-p, & \text{if}\ p\leq\tilde{p}\leq P^r(t) \\
			P^r(t)-p, & \text{if}\ p\leq P^r(t)\leq\tilde{p} \\
			0, & \text{if}\ P^r(t)\leq p\leq\tilde{p},
			\end{array}\right.
			\end{equation}
			which covers all possible cases as a result of the condition $p\leq\tilde{p}$. Hence, the $E$-$p$ transform of the saturated signal can be found as
			\begin{equation}
			\begin{aligned}
			E_{\bar{P}^r}(p)
			&=\int_0^\infty\max\{\min\{P^r(t),\tilde{p}\}-p,0\}dt \\
			&=\int_0^\infty\big[\max\{P^r(t)-p,0\}-\max\{P^r(t)-\tilde{p},0\}\big]dt \\
			&=E_{P^r}(p)-E_{P^r}(\tilde{p}) \\
			&=E_{P^r}(p)-\Phi_{\bar{p},x}(P^r)\ \ \forall p\leq\tilde{p}.
			\end{aligned}
			\end{equation}
			This expression is non-negative because, from the cumulative nature of the $E$-$p$ transform, $E_{P^r}(p)\geq E_{P^r}(\tilde{p})$ for all $p$ with $0\leq p\leq\tilde{p}$.
			
			Now consider instead the region defined by $p>\tilde{p}$, for which we have, from the definition of the saturation signal, that $E_{\bar{P}^r}(p)=0$. Hence, the general from of the adjusted $E$-$p$ transform is
			\begin{equation}
			E_{\bar{P}^r}(p)=\max\{E_{P^r}(p)-\Phi_{\bar{p},x}(P^r),0\},
			\label{eq:E_tildeP}
			\end{equation}
			therefore $E_{\bar{P}^r}(p)\leq\Omega_{\bar{p},x}(p)\ \forall p$ and so the saturated signal is feasible (see \cite{Evans2018} for further details). 
			
			Now, in general, if $\check{P}^r(\cdot)$ is implemented as a feasible approximation to $P^r(\cdot)$, then the energy-not-served can be calculated as $\Gamma(\check{P}^r,P^r)=E_{P^r}(0)-E_{\check{P}^r}(0)$. Evaluation of \eqref{eq:E_tildeP} at $p=0$ leads to
			\begin{equation}
			\Gamma(\bar{P}^r,P^r)=E_{P^r}(0)-E_{\bar{P}^r}(0)=\Phi_{\bar{p},x}(P^r),
			\end{equation}
			hence the energy-not-served under the saturated signal is equal to the max energy gap. 
			\qed

		\subsection{Proof of \Cref{Theorem V.6}}
		\begin{lemma}
			\label{lem fgap1}
			For any given reference $P^r(\cdot)$, the max energy gap forms an upper bound on the minimum energy-not-served, i.e. $\Gamma_{\bar{p},x}^*(P^r)\leq\Phi_{\bar{p},x}(P^r)$.
		\end{lemma}
		\begin{proof}
			Follows directly from Lemma~\ref{Lemma V.5}; if any signal exists with such a value of energy-not-served, this must form an upper bound on the minimum value that this quantity can take.
		\end{proof}
		\begin{lemma}
			\label{lem fgap2}
			For any given reference $P^r(\cdot)$, the max energy gap forms a lower bound on the minimum energy-not-served, i.e. $\Gamma_{\bar{p},x}^*(P^r)\geq\Phi_{\bar{p},x}(P^r)$.
		\end{lemma}	
		\begin{proof}
			Denote by $P^{r*}(\cdot)\in\{\check{P}^r(\cdot)\colon\Gamma(\check{P}^r,P^r)=\Gamma_{\bar{p},x}^*\}$ any signal resulting in the minimum ENS; and note that this must be feasible by construction. Defining $p^*$ as any power level satisfying $p^*{\in} \arg\max_{p\geq0}\big[E_{P^r}(p)-\Omega_{\bar{p},x}(p)\big]$, we know that $E_{P^r}(p^*){=}\Omega_{\bar{p},x}(p^*){+}\Phi_{\bar{p},x}(P^r)$ from the definition of $\Phi_{\bar{p},x}(P^r)$. Hence 
			\begin{equation}
			E_{P^{r*}}(p^*)\leq\Omega_{\bar{p},x}(p^*)=E_{P^r}(p^*)-\Phi_{\bar{p},x}(P^r).
			\end{equation}
			Moreover, for almost all $p$,
			\begin{equation}
			\begin{aligned}
			\frac{d}{dp}E_{P^{r*}}(p)&=\frac{d}{dp}\int_0^\infty \textnormal{max}\big\{P^{r*}(t)-p,0\big\}dt\\
			&=-\mu\big(\big\{\tau\colon P^{r*}(\tau)\geq p\big\}\big) \\
			&\geq-\mu\big(\big\{\tau\colon P^r(\tau)\geq p\big\}\big) 
			=\frac{d}{dp}E_{P^{r}}(p),
			\end{aligned}
			\label{eq:ddp_comp}
			\end{equation}
			hence integration of \eqref{eq:ddp_comp} with respect to $p\leq p^*$ yields
			\begin{equation}
			\begin{aligned}
			E_{P^{r*}}(p)
			&=E_{P^{r*}}(p^*)-\int_{p}^{p^*}\frac{d}{dp}E_{P^{r*}}(\pi)d\pi\\
			&\leq E_{P^{r*}}(p^*)-\int_{p}^{p^*}\frac{d}{dp}E_{P^{r}}(\pi)d\pi\\
			&\leq E_{P^{r}}(p^*)-\Phi_{\bar{p},x}(P^r)-\int_{p}^{p^*}\frac{d}{dp}E_{P^{r}}(\pi)d\pi\\
			&=E_{P^{r}}(p)-\Phi_{\bar{p},x}(P^r)\ \ \forall p\leq p^*.
			\end{aligned}
			\end{equation}
			Evaluation of this condition at $p=0$ leads to the result.
		\end{proof}
	
		\vspace{-4mm}
		\subsection{Proof of \Cref{Lemma VI.1}}
			\noindent Firstly, consider the case that $0<P^r<\sum_{i=1}^n\bar{p}_i\min\{x_i,\Delta t\}$. Defining
			\begin{equation}
			\hat{x}\doteq\min_{i\in\{1,...,n\}}\big\{z_i\colon z_i<x_i\big\},
			\end{equation}
			we are able to deduce that a positive value of $\hat{x}$ must exist, and moreover that any arbitrary device $D_i$ falls into one of the following three categories at the final time:			
			\begin{enumerate}
				\item $D_i{\in}\{D_j{\colon} z_j{=}\hat{x}\}$, in which case $z_i{=}\hat{x}$.
				\item $D_i{\in}\{D_j{\colon} z_j{>}\hat{x}\}$, in which case it has been run at maximum power for the entire time interval, hence $z_i{=}x_i{-}\Delta t$.
				\item $D_i{\in}\{D_j{\colon} z_j{<}\hat{x}\}$, in which case it has not been run at any time within the interval and so $z_i{=}x_i$.
			\end{enumerate}
			Hence,
				\begin{equation}
				z_i=\min\{\max\big\{x_i-\Delta t,\hat{x}\big\},x_i\} \label{eq:z_i}
				\end{equation}
			with $0<\hat{x}<\max\{x_i,\ i=1,...,n\}$.

			Next, consider the case that $P^r=0$, for which the final state is trivially $z=x$, or equivalently \eqref{eq:z_i} with $\hat{x}=\max\{x_i,\ i=1,...,n\}$.
			
			Finally, consider the remaining case, for which $z_i=\max\{x_i-\Delta t,0\}$, or equivalently \eqref{eq:z_i} with $\hat{x}=0$. Thus the result follows as a combination of the three cases.
		\qed
		
		\vspace{-2mm}
		\subsection{Proof of \Cref{Lemma VI.2}}
		\noindent It follows by construction that this choice of input returns the final state. Moreover, this input can be written as
		\begin{equation}
		\begin{aligned}
		(t_2-t_1)u&=P(x-z) 
		=-P\int_{t_1}^{t_2}\dot{z}(\tau)d\tau \\
		&=-P\int_{t_1}^{t_2}\big[-P^{-1}u^*(\tau)\big]d\tau 
		=\int_{t_1}^{t_2}u^*(\tau)d\tau,
		\end{aligned}
		\end{equation}
		in which $u^*(\cdot)$ denotes the input allocated according to \eqref{eq: explicit policy}. Hence the constant input is an averaged form of the optimal input and must therefore be feasible.
		\qed

	\bibliographystyle{IEEEtran}
	\bibliography{biblio3,biblio_manual2}
\end{document}